\documentclass[11pt,a4paper]{article}

\usepackage{mathrsfs}
\usepackage{amscd}
\usepackage{hyperref}
\hypersetup{colorlinks=true,citecolor=green,linkcolor=blue,urlcolor=blue}
\usepackage{tipa}
\usepackage{amssymb}
\usepackage{amsfonts}
\usepackage{stmaryrd}
\usepackage{amssymb}

\usepackage{CJK}  
\usepackage{indentfirst} 

\usepackage{latexsym}   
\usepackage{bm}         

\usepackage{pifont}
\usepackage{bm}
\usepackage{amsmath,amssymb,amsfonts}
\usepackage{indentfirst}

\usepackage{xcolor}
\usepackage{cases}
\usepackage{wasysym}
\usepackage{amsthm}

\newtheorem{theorem}{Theorem}[section]
\newtheorem{lemma}[theorem]{Lemma}
\newtheorem{proposition}[theorem]{Proposition}

\newtheorem{remark}[theorem]{Remark}
\newtheorem{corollary}[theorem]{Corollary}
\newtheorem{definition}[theorem]{Definition}
\CJKtilde   

\begin{document}

\title{\Large \textbf{On certain Iwahori representations of unramified $U(2, 1)$ in characteristic $p$}}
\date{}
\author{\textbf{Peng Xu}}
\maketitle

\begin{abstract}
Let $F$ be a non-archimedean local field of odd residue characteristic $p$. Let $G$ be the unramified unitary group $U(2, 1)(E/F)$, and $K$ be a maximal compact open subgroup of $G$. For an $\overline{\mathbf{F}}_p$-smooth representation $\pi$ of $G$ containing a weight $\sigma$ of $K$, we follow the work of Hu (\cite{Hu12}) to attach $\pi$ a certain $I_K$-subrepresentation, where $I_K$ is the Iwahori subgroup in $K$. In terms of such an $I_K$-subrepresentation, we prove a sufficient condition for $\pi$ to be non-finitely presented. We determine such an $I_K$-subrepresentation explicitly, when $\pi$ is either a spherical universal Hecke module or an irreducible principal series.
\end{abstract}

\tableofcontents

\section{Introduction}
In the last fifteen years, the area of $p$-modular representations of $p$-adic reductive groups is in a period of vast development. The recent work of Abe--Henniart--Herzig--Vign$\acute{\text{e}}$ras (\cite{AHHV17}) and their forthcoming ones, generalizing \cite{Her2011b}, \cite{Abe-2011}, reduce the classification of irreducible admissible mod-$p$ representations of a $p$-adic reductive group to supersingular (i.e., supercuspidal) representations, which is similar to the classical work of Bernstein and Zelevinski on the classification of complex smooth representations of $GL_n$ (\cite{BZ1977}).

However, supersingular representations remain mysterious largely since Barthel and Livn$\acute{\text{e}}$ discovered them two decades ago, and the classifications are only understood for the group $GL_2 (\mathbf{Q}_p)$ (\cite{Breuil03}) and a few other groups closely related to it. In general, the work of Breuil and Pa$\check{\text{s}}$k$\bar{\text{u}}$nas (\cite{BP12}) shows that there are much more supersingular representations of $GL_2 (\mathbf{Q}_{p^f})$ ($f> 1$) than the two dimensional irreducible continuous mod-$p$ representations of the absolute Galois group $G_{\mathbf{Q}_{p^f}}$, and their method is to construct many supersingular representations in the $0$-th homology group of certain coefficient systems attached to the Bruhat--Tits tree of $SL_2$, where the recipes of coefficient systems in use come from the weight part of generalized Serre's conjecture (\cite{BDJ10}).

To analyze the Bruhat--Tits building of the group in consideration is then very useful; actually in most works mentioned above, a maximal compact induction and its associated spherical Hecke algebra play crucial roles. In \cite{Hu12}, Hu attached a diagram \footnote{Roughly speaking, a diagram is the restriction of a $G$-equivariant coefficient system to a fixed edge on the tree of $G$.} to an irreducible smooth representation (with central character) of $GL_2$, and he proved such a diagram determines the original representation uniquely. Hu has also determined his canonical diagrams explicitly in many important cases.

In the current paper, we follow Hu's idea to study an analogous problem for the unitary group $G= U(2, 1)(E/F)$ over a non-archimedean local field $F$ of odd residue characteristic $p$. Let $K$ be a maximal compact open subgroup of $G$, and $\sigma$ be an irreducible smooth representation of $K$ over $\overline{\mathbf{F}}_p$. By considering the Bruhat--Tits tree of $G$, the maximal compact induction $\textnormal{ind}^G _K \sigma$ is decomposed into a sum of $I_K$-representations
\begin{center}
$\text{ind}^G _K \sigma= I^+ (\sigma) \oplus I^- (\sigma)$,
\end{center}
where $I_K$ is the Iwahori subgroup in $K$. For a smooth representation $\pi$ of $G$ containing $\sigma$, we consider the intersection $I^{+} (\sigma, \pi)\cap I^{-} (\sigma, \pi)$ of the images of $I^+ (\sigma)$ and $I^- (\sigma)$ in $\pi$, which by definition is an $I_K$-subrepresentation of $\pi$. Such an $I_K$-subrepresentation is expected to contain important information of $\pi$.

The first main result proved is as follows:
\begin{theorem}\label{first main intro}(Theorem \ref{images of f_0 and f_1 give that of I_K})
Assume $\pi=\textnormal{ind}^G _K \sigma /(T-\lambda)$, for a $\lambda \in \overline{\mathbf{F}}_p$. Then the representation $I^{+} (\sigma, \pi)\cap I^{-} (\sigma, \pi)$ is two dimensional, with a canonical basis.
\end{theorem}
Here, the notation $T$ denotes certain spherical Hecke operator (see subsection \ref{subsection 2.2}). The representation $\pi$ considered in above theorem is usually called a spherical universal Hecke module, and such a representation plays a central role in the $p$-modular representation theory of $G$.

The second main result proved is the following:
\begin{theorem}\label{second main intro}(Theorem \ref{main})
Assume $\pi$ is an irreducible principal series. Then
 \begin{center}
 $I^{+} (\sigma, \pi)\cap I^{-} (\sigma, \pi)=\pi^{I_{1, K}}$,
 \end{center}
where $I_{1, K}$ is the pro-$p$-Sylow subgroup of $I_{K}$.
\end{theorem}

Both theorems above are analogy of results of Hu on $GL_2$ (\cite{Hu12}). In the case of $GL_2$, such an Iwahori subrepresentation is the main ingredient in Hu's canonical diagram attached to $\pi$. In this paper, we don't define a diagram explicitly but only keep it in mind as a general guideline.

We also obtain some other partial result on the $I_K$-subrepresentation $I^{+} (\sigma, \pi)\cap I^{-} (\sigma, \pi)$:
\begin{proposition}\label{hu's criteria intro}(Proposition \ref{Hu's criteria})
If $\pi$ is finitely presented, then the $I_K$-subrepresentation $I^{+} (\sigma, \pi)\cap I^{-} (\sigma, \pi)$ is finite dimensional.
\end{proposition}

This Proposition is an analogue of one direction of Hu's criteria for finite presentation of smooth representations of $GL_2 (F)$. Note that Hu's criteria has been crucially used in his work (for $F$ of positive characteristic) and Schraen's work (for quadratic extensions $F/\mathbf{Q}_p$) on non-finite presentation of supersingular representations of $GL_2 (F)$ (\cite{Hu12}, \cite{Sch15}).

\medskip

This paper is organized as follows. In section \ref{section: notation}, we fix notations, and recall (actually prove) some preliminary results. In the first part of section \ref{section: spheres}, we recall a natural splitting of the spherical Hecke operator $T$ and prove that it satisfies some good properties, and in the second part we exhaust certain computation on the tree of $G$. In early parts of section \ref{section: initial results on the I-reps}, we prove the $I_K$-subrepresentation $I^+ (\sigma, \pi)\cap I^- (\sigma, \pi)$ is always non-zero for certain $\pi$, and prove Proposition \ref{hu's criteria intro}, where in later parts of this section we record some technical results and prove conditionally that the representation $I^+ (\sigma, \pi)\cap I^- (\sigma, \pi)$ is independent of the choice of $\sigma$. In section \ref{section: main results}, we prove the main Theorems and some other related results.

\section{Notations and Preliminary results}\label{section: notation}

The first two subsections reproduce almost \cite[Section 2]{X2016}.
\subsection{Notations}
Let $F$ be a non-archimedean local field of odd residue characteristic $p$, with ring of integers $\mathfrak{o}_F$ and maximal ideal $\mathfrak{p}_F$, and let $k_F$ be its residue field of cardinality $q=p^f$. Fix a separable closure $F_s$ of $F$. Let $E$ be the unramified quadratic extension of $F$ in $F_s$. We use similar notations $\mathfrak{o}_E$, $\mathfrak{p}_E$, $k_E$ for analogous objects of $E$, and we denote by $E^1$ the norm $1$ subgroup of $E^\times$. Let $\varpi_{E}$ be a uniformizer
of $E$, lying in $F$. Given a 3-dimensional vector space $V$ over $E$, we identify it with $E^{3}$, by fixing a basis of $V$.
Equip $V$ with the non-degenerate Hermitian form h:
\begin{center}
 $\text{h}:~V \times V \rightarrow E$,
$(v_{1}, v_{2}) \mapsto ~v_1 ^{\text{T}}\beta \overline{v_2}, v_1, v_2\in V$.
\end{center}
Here, $-$ denotes the non-trivial Galois conjugation on $E/F$, inherited by $V$, and
$\beta$ is the matrix
\[ \begin{matrix}\begin{pmatrix} 0  & 0 & 1  \\ 0  & 1 & 0\\
1 & 0 & 0
\end{pmatrix}
\end{matrix}. \]
The unitary group $G$ is the subgroup of $GL(3, E)$ whose elements fix the Hermitian form h:

\begin{center}
$G=\{g\in \text{GL}(3, E)\mid \text{h}(gv_{1}, gv_{2})= \text{h}(v_{1}, v_{2}), \text{for~any}~v_{1}, v_{2}~\in V\}.$
\end{center}

Let $B=HN$ (resp, $B'= HN'$) be the subgroup of upper (resp, lower) triangular matrices of $G$, where $N$ (resp, $N'$) is the unipotent radical of $B$ (resp, $B'$) and $H$ is the diagonal subgroup of $G$. Denote an element of the following form in $N$ and $N'$ by $n(x, y)$ and $n'(x, y)$ respectively:
\begin{center}
$\begin{pmatrix}  1 & x & y  \\ 0 & 1 & -\bar{x}\\
0 & 0 & 1
\end{pmatrix}$, ~
$\begin{pmatrix}  1 & 0 & 0   \\ x & 1 & 0\\
y & -\bar{x} & 1
\end{pmatrix}$
\end{center}
where $(x, y)\in E^2$ satisfies $x\bar{x}+ y+ \bar{y}=0$. Denote by $N_k$ (resp, $N'_k$), for any $k\in \mathbb{Z}$, the subgroup of $N$ (resp, $N'$) consisting of $n(x, y)$ (resp, $n'(x, y)$) with $y\in \mathfrak{p}^{k}_E$. For $x\in E^\times$, denote by $h(x)$ an element in $H$ of the following form:
\begin{center}
$\begin{pmatrix}  x & 0 & 0  \\ 0 & -\bar{x}x^{-1} & 0\\
0 & 0 & \bar{x}^{-1}
\end{pmatrix}$
\end{center}

We record the following useful identity in $G$: for $y\in E^\times$,
\begin{equation}\label{useful identity}
\beta n(x, y)= n(\bar{y}^{-1}x, y^{-1})\cdot h(\bar{y}^{-1})\cdot n'(-\bar{y}^{-1}\bar{x}, y^{-1}).
\end{equation}

\medskip
Up to conjugacy, the group $G$ has two maximal compact open subgroups $K_0$ and $K_1$ (\cite{Hij63}, \cite[Section 2.10]{Tits79}), which are given by:
\begin{center}
$K_0= \begin{pmatrix}  \mathfrak{o}_E & \mathfrak{o}_E & \mathfrak{o}_E  \\ \mathfrak{o}_E  & \mathfrak{o}_E & \mathfrak{o}_E\\
\mathfrak{o}_E & \mathfrak{o}_E & \mathfrak{o}_E
\end{pmatrix}\cap G, ~K_1= \begin{pmatrix}  \mathfrak{o}_E & \mathfrak{o}_E & \mathfrak{p}^{-1}_E  \\ \mathfrak{p}_E  & \mathfrak{o}_E & \mathfrak{o}_E\\
\mathfrak{p}_E & \mathfrak{p}_E & \mathfrak{o}_E
\end{pmatrix}\cap G$
\end{center}
The maximal normal pro-$p$ subgroups of $K_0$ and $K_1$ are respectively:

$K^1 _0= 1+\varpi_E M_3 (\mathfrak{o}_E)\cap G, ~K^1 _1= \begin{pmatrix}  1+\mathfrak{p}_E & \mathfrak{o}_E & \mathfrak{o}_E  \\ \mathfrak{p}_E  & 1+\mathfrak{p}_E & \mathfrak{o}_E\\
\mathfrak{p}^2_E & \mathfrak{p}_E & 1+\mathfrak{p}_E
\end{pmatrix}\cap G$

Let $\alpha$ be the following diagonal matrix in $G$:
\[ \begin{matrix}\begin{pmatrix} \varpi_{E}^{-1}  & 0 & 0  \\ 0  & 1 & 0\\
0 & 0 & \varpi_{E}
\end{pmatrix}
\end{matrix} ,\]
and put $\beta'=\beta \alpha^{-1}$. Note that $\beta\in K_0$ and $\beta'\in K_1$. We use $\beta_K$ to denote the unique element in $K\cap \{\beta, \beta'\}$.

Let $K\in \{K_0, K_1\}$, and $K^1$ be the maximal normal pro-$p$ subgroup of $K$. We identify the finite group $\Gamma_K= K/K^1$ with the $k_F$-points of an algebraic group defined over $k_F$, denoted also by $\Gamma_K$: when $K$ is $K_0$, $\Gamma_K$ is $U(2, 1)(k_E /k_F)$, and when $K$ is $K_1$, $\Gamma_K$ is $U(1, 1)\times U(1)(k_E /k_F)$. Let $\mathbb{B}$ (resp, $\mathbb{B}'$) be the upper (resp, lower) triangular subgroup of $\Gamma_K$, and $\mathbb{U}$ (resp, $\mathbb{U}'$) be its unipotent radical. The Iwahori subgroup $I_K$ (resp, $I'_K$) and pro-$p$ Iwahori subgroup $I_{1,K}$ (resp, $I' _{1,K}$) in $K$ are the inverse images of $\mathbb{B}$ (resp, $\mathbb{B}'$) and $\mathbb{U}$ (resp, $\mathbb{U}'$) in $K$. We have the following Bruhat decomposition for $K$:
 \begin{center}
 $K= I_K \cup I_K \beta_K I_K$.
 \end{center}
\medskip

\medskip
All the representations of $G$ and its subgroups considered in this paper are smooth over $\overline{\mathbf{F}}_p$.

\subsection{The spherical Hecke algebra $\mathcal{H}(K, \sigma)$}\label{subsection 2.2}
 Let $K$ be a maximal compact open subgroup of $G$, and $(\sigma, W)$ be an irreducible smooth representation of $K$. As $K^1$ is pro-$p$ and normal, $\sigma$ factors through the finite group $\Gamma_K= K/K^1$, i.e., $\sigma$ is the inflation of an irreducible representation of $\Gamma_K$. Conversely, any irreducible representation of $\Gamma_K$ inflates to an irreducible smooth representation of $K$. We may therefore identify irreducible smooth representations of $K$ with irreducible representations of $\Gamma_K$, and we shall call them \emph{weights} of $K$ or $\Gamma_K$ from now on.

  It is known that $\sigma^{I_{1,K}}$ and $\sigma_{I'_{1,K}}$ are both one-dimensional, and that the natural composition map $\sigma^{I_{1,K}}\hookrightarrow \sigma \twoheadrightarrow \sigma_{I'_{1,K}}$ is non-zero, i.e., an isomorphism of vector spaces (\cite[Theorem 6.12]{C-E2004}). Denote by $j_\sigma$ the inverse of the composition map just mentioned. For $v\in \sigma^{I_{1,K}}$, we have $j_\sigma (\bar{v})= v$, where $\bar{v}$ is the image of $v$ in $\sigma_{I'_{1,K}}$. When viewed as a map in $\text{Hom}_{\overline{\mathbf{F}}_p}(\sigma, \sigma^{I_{1, K}})$, the $j_\sigma$ factors through $\sigma_{I'_{1,K}}$, i.e., it vanishes on $\sigma (I'_{1,K})$.

\begin{remark}\label{value of lambda}
There is a unique constant $\lambda_{\beta_K, \sigma}\in \overline{\mathbf{F}}_p$, such that $\beta_K \cdot v-\lambda_{\beta_K, \sigma}v\in \sigma (I'_{1,K})$, for $v\in \sigma^{I_{1,K}}$. The value of $\lambda_{\beta_K, \sigma}$ is known: it is zero unless $\sigma$ is a character (\cite[Proposition 3.16]{H-V2011}), due to the fact that $\beta_K\notin I_K\cdot I'_K$. When $\sigma$ is a character, $\lambda_{\beta_K, \sigma}$ is just the scalar $\sigma (\beta_K)$.
\end{remark}

\begin{remark}\label{n_K and m_K}
There are unique integers $n_K$ and $m_K$ such that $N\cap I_{1, K}= N_{n_K}$ and $N'\cap I_{1, K}=N'_{m_K}$. Note that $n_K +m_K= 1$.
\end{remark}

 Let $\text{ind}_K ^{G}\sigma$ be the compactly induced smooth representation, i.e., the representation of $G$ with underlying space $S(G, \sigma)$
\begin{center}
$S(G, \sigma)=\{f: G\rightarrow W\mid  f(kg)=\sigma (k)\cdot f(g),~\text{for~any}~k\in K~\text{and}~g\in G,~ \text{locally~constant~with~compact~support}\}$
\end{center}
and $G$ acting by right translation.

As usual (\cite[section 2.3]{B-L95}), denote by $[g, v]$ the function in $S(G, \sigma)$, supported on $Kg^{-1}$ and having value $v\in W$ at $g^{-1}$.  An element $g'\in G$ acts on the function $[g, v]$ by $g'\cdot[g, v]= [g'g, v]$, and we have $[gk, v]= [g, \sigma(k)v]$ for $k\in K$.

The spherical Hecke algebra $\mathcal{H}(K, \sigma)$ is defined as $\text{End}_G (\text{ind}^G _K \sigma)$, and by \cite[Proposition 5]{B-L95} it is isomorphic to the convolution algebra $\mathcal{H}_K (\sigma)$ of all compactly support and locally constant functions $\varphi$ from $G$ to $\text{End}_{\overline{\mathbf{F}}_p}(\sigma)$, satisfying $\varphi(kgk')=\sigma(k)\varphi(g)\sigma(k')$ for any $g\in G$ and $k, k'\in K$. Let $\varphi$ be the function in $\mathcal{H}_K (\sigma)$, supported on $K\alpha K$, and satisfying $\varphi (\alpha)= j_\sigma$. Denote by $T$ the Hecke operator in $\mathcal{H}(K, \sigma)$, which corresponds to the function $\varphi$, via the aforementioned isomorphism between $\mathcal{H}_K (\sigma)$ and $\mathcal{H}(K, \sigma)$. We refer the readers to \cite[(4)]{X2016} for a formula of $T$.

The following proposition is a special case of \cite[Corollary 1.3]{Her2011a}.

\begin{proposition}
The algebra $\mathcal{H}(K, \sigma)$ is isomorphic to $\overline{\mathbf{F}}_p [T]$.
\end{proposition}

\subsection{A canonical set of generators for $\sigma$}

Recall the following Iwahori decomposition
\begin{equation}\label{iwahori decom in K}
K = I_K\bigcup_{u\in N_{n_K}/N_{n_K +1}} [u]\beta_K I_K,
\end{equation}
where $[u]$ denotes a representative of $u \in N_{n_K}/N_{n_K +1}$ in $N_{n_K}$.

\medskip
By the above \eqref{iwahori decom in K}, a set of representatives for the coset space $K/ I_K$ is given by
\begin{center}
$\{Id \} \cup \{[u]\beta_K\mid u\in  N_{n_K}/N_{n_K +1}\}$.
\end{center}
Let $v_0$ be a non-zero vector in the line $\sigma^{I_{1, K}}$. The vector $v_0$ generates the weight $\sigma$. Thus, the set
\begin{center}
$\{v_0\} \cup \{u\beta_K v_0\mid u\in  N_{n_K}/N_{n_K +1}\}$,
\end{center}
or alternatively the set
\begin{center}
$\{\beta_K v_0\} \cup \{\beta_K u\beta_K v_0\mid u\in  N_{n_K}/N_{n_K +1}\}$
\end{center}
spans the underlying space of $\sigma$.

However, we may do a little better:
\begin{lemma}\label{generators of sigma}
The set
\begin{center}
$\{u\beta_K v_0\mid u\in  N_{n_K}/N_{n_K +1}\}$,
\end{center}
or equivalently the set
\begin{center}
$\{\beta_K u\beta_K v_0\mid u\in  N_{n_K}/N_{n_K +1}\}$
\end{center}
spans the underlying space of $\sigma$.
\end{lemma}
\begin{proof}
If the weight $\sigma$ is a one-dimensional character, the statement is clear. Assume $\text{dim}~\sigma > 1$. the statement follows from the preceding remark and Remark \ref{value of lambda}. More precisely, due to the fact $\beta_K \notin I_K\cdot I'_K$, we have:
\begin{center}
$\beta_K v_0 \in \sigma (I'_{1, K})$.
\end{center}
Note that $I'_{1, K} =(\beta_K N_{n_K} \beta_K) K^1$, and that the group $I'_{1, K}$ fixes $\beta_K v_0$. We have immediately that $\sigma (I'_{1, K})$ is contained in the subspace of $\sigma$ spanned by the set
\begin{center}
$\{\beta_K u\beta_K v_0\mid u\in  N_{n_K}/N_{n_K +1}\}$.
\end{center}
We are done.
\end{proof}

\begin{remark}
One can also prove the Lemma by applying \cite[Lemma 2.10]{Hu12} to our case.
\end{remark}

\begin{remark}
When $\sigma$ is the Steinberg weight, its dimension is exactly the order of the coset space $N_{n_K}/N_{n_K +1}$, and the Lemma gives a \textbf{canonical basis} of it.
\end{remark}

\subsection{The space $(\textnormal{ind}_K ^G \sigma)^{I_{1, K}}$ and its image in $\textnormal{ind}^G _K \sigma / (T-\lambda)$}

We fix a non-zero vector $v_0\in \sigma^{I_{1, K}}$. Let $f_n$ be the function in $(\textnormal{ind}_K ^G \sigma)^{I_{1, K}}$, supported on $K \alpha^{-n} I_{1, K}$, such that
\begin{center}
$f_n (\alpha^{-n})=  \begin{cases}
\beta_K\cdot v_0, ~~~~~~~n>0,\\
v_0 ~~~~~n\leq 0.
\end{cases}$
\end{center}
Then, we have the following (\cite[Lemma 3.5]{X2016})

\begin{lemma}

The set of functions $\{f_n \mid n\in \mathbb{Z}\}$ consists of a basis of the $I_{1, K}$-invariants of the maximal compact induction $\textnormal{ind}^{G} _K\sigma$.

\end{lemma}

The following proposition (\cite[Proposition 3.6, Corollary 3.11]{X2016}) is very useful:
\begin{proposition}\label{hecke operator formula}

We have:

$(1)$~~~$T\cdot f_0= f_{-1} + \lambda_{\beta_K, \sigma}\cdot f_1$.

$(2)$~~~For $n\neq 0$, $T \cdot f_n= cf_n +f_{n+\delta(n)}$, where $c$ is a constant (depending on $\sigma$) and $\delta(n)$ is either $1$ or $-1$, depending on $n> 0$ or $< 0$.
\end{proposition}

We record some simple and useful corollaries.
\begin{corollary}\label{injectivity of hecke operator}
Any non-zero Hecke operator $P(T)$ is injective.
\end{corollary}
\begin{proof}
The kernel of $P(T)$ is $I_{1, K}$-stable; if it is non-zero, it contains some non-zero $I_{1, K}$-invariant function (\cite[Lemma 1]{B-L95}), which is a linear combination of the functions $\{f_n\}_{n\in \mathbb{Z}}$. But that can not happen by Proposition \ref{hecke operator formula}.
\end{proof}

\begin{corollary}\label{f_0 and f_1 are l.d in quotient}
For any $\lambda \in \overline{\mathbf{F}}_p$, the image of the space $(\textnormal{ind}^G _K \sigma)^{I_{1, K}}$ in the representation $\textnormal{ind}^G _K \sigma / (T-\lambda)$ is two dimensional, generated by the images of the functions $f_0$ and $f_1$.
\end{corollary}

\begin{proof}
Applying $(2)$ of Proposition \ref{hecke operator formula} repeatedly, we see $f_n \in \langle f_1\rangle_{\overline{\mathbf{F}}_p} +(T-\lambda)$ for $n\geq 2$, and  $f_{-k}\in \langle f_{-1}\rangle_{\overline{\mathbf{F}}_p}+ (T-\lambda)$ for $k\geq 2$. Using $(1)$ of Proposition \ref{hecke operator formula}, we see that $f_{-1}\in \langle f_0, f_1\rangle_{\overline{\mathbf{F}}_p} +(T-\lambda)$. It remains to check that $f_0 -c f_1 \notin (T-\lambda)$ for any $c\in \overline{\mathbf{F}}_p$. If there is a non-zero function $f$ so that
\begin{center}
$f_0- cf_1 =(T-\lambda)f$
\end{center}
holds for some $c$, by last Corollary the function $f$ itself must be $I_{1, K}$-invariant; however, Proposition \ref{hecke operator formula} implies that such an equality can not hold for any non-zero $I_{1, K}$-invariant $f$.
\end{proof}

\section{The spheres $C_{n, \sigma}$}\label{section: spheres}

\subsection{The spheres $C_{n, \sigma}$ and their images under $T$}
Let $K$ be a maximal compact open subgroup of $G$, and $\sigma$ be a weight of $K$. For $n\geq 0$, denote by $R^+ _n (\sigma)$ (resp, $R^{-} _n (\sigma)$) the subspace of functions in $\text{ind}^G _K \sigma$ supported in the coset $K\alpha^n I_K$ (resp, $K\alpha^{-(n+1)}I_K$). Both spaces are $I_K$-stable.

\begin{definition}(\cite[Lemma 3.7]{X2016})

For $n\geq 0$, $R^+ _n (\sigma)= [N_{n_K} \alpha^{-n}, \sigma]$;

For $n\geq 1$, $R^- _{n-1} (\sigma) =[N'_{m_K}\alpha^n,  \sigma]$.

\end{definition}
Put $R_0 (\sigma)= R^- _{-1} (\sigma)=R^+ _0 (\sigma)$, and $R_n (\sigma)= R^+ _n (\sigma)\oplus R^- _{n-1} (\sigma)$, for $n\geq 1$. In terms of the tree of $G$, the space $R_n (\sigma)$ consists of all functions in $\text{ind}^G _K \sigma$ supported in the cycle $\mathbf{C}_n$ of radius $2n$ from the vertex $v_K$, and for this reason we will also denote it by $C_{n, \sigma}$.

\begin{remark}\label{Question}
 It is well-known that the $K$-space $R_0 (\sigma)$ is generated by its one-dimensional $I_{1, K}$-invariants. But we point out this is no longer true for $R_n (\sigma)$ $(n\geq 1)$: recall that the subspace of $I_{1, K}$-invariants of $R_n (\sigma) (n\geq 1)$ is two-dimensional, with a basis $\{f_{-n}, f_{n}\}$ (\cite[Remark 3.8]{X2016}). An estimation of the dimension of the subrepresentation generated by these two functions shows that it is strictly smaller than that of $R_n (\sigma)$.
\end{remark}

\medskip
The following Proposition (\cite[Proposition 3.9]{X2016}) describes how the above $I_K$-stable spaces are changed under the Hecke operator $T$.
\begin{proposition}\label{image of the tree under T}

$(1)$.~~$T (R^+ _0 (\sigma)) \subseteq R^{+}_1 (\sigma)\oplus R^{-}_0 (\sigma)$.

$(2)$.~~$T (R^{+}_n (\sigma)) \subseteq R^{+}_{n-1} (\sigma)\oplus R^{+}_n (\sigma)\oplus R^{+}_{n+1} (\sigma), n \geq 1$.

$(3)$.~~ $T (R^{-}_n (\sigma))\subseteq R^{-}_{n-1} (\sigma)\oplus R^{-}_{n} (\sigma)\oplus R^{-}_{n+1} (\sigma), n\geq 0$.

\end{proposition}

\medskip

By $(3)$ of Proposition \ref{image of the tree under T}, for $n\geq 0$, we may write $T\mid_{R^- _n (\sigma)}$ as the sum of two operators $T^-$ and $T^+$, where $T^-: R^- _n (\sigma)\rightarrow R^- _{n-1} (\sigma)$, and $T^+: R^- _n (\sigma)\rightarrow R^- _n (\sigma) \oplus R^- _{n+1} (\sigma)$. Similarly, from $(2)$ of Proposition \ref{image of the tree under T}, for $n \geq 1$, we may also write $T\mid_{R^+ _n (\sigma)}$ as the sum of $T^-$ and $T^+$,
where $T^-: R^+ _n (\sigma)\rightarrow R^+ _{n-1} (\sigma)$, and $T^+: R^+ _n (\sigma)\rightarrow R^+ _n (\sigma) \oplus R^+ _{n+1} (\sigma)$. Both operators $T^-$ and $T^+$ are $I_K$-maps.

We record the formula of $T^-$ here, which is very simple, and is implicit in the argument of \cite[Proposition 3.9]{X2016}:
\begin{center}
$T^- f= \begin{cases}
 [u'\alpha^n, j_\sigma v],~~~~~~~~~~~~~~~~~~~~~\text{if}~f=[u'\alpha^{n+1}, v]\in R^- _n (\sigma).\\
 [u\alpha^{-(n-1)}, \beta_K j_\sigma \beta_K v],~~~\text{if}~f=[u\alpha^{-n}, v]\in R^+ _n (\sigma).
\end{cases}$
\end{center}

\begin{corollary}\label{injectivity of T_- and surjectivity of T+}
The operator $T^-$ is surjective, and $T^+$ is injective.
\end{corollary}

\begin{proof}

The surjection of $T^-$ follows from its formula above and Lemma \ref{generators of sigma}. In more words, if we take $v= v_0$ and $\beta_K v_0$ respectively in the above formulae, we get:
\begin{center}
$T^- [u'\alpha^{k+1}, v_0]= [u'\alpha^k, v_0]$,
\end{center}
and
\begin{center}
$T^- [u\alpha^{-k}, \beta_K v_0]= [u\alpha^{-(k-1)}, \beta_K v_0]$.
\end{center}

We observe that, for $u \in N_{n_K}$
\begin{center}
$\alpha^k \beta_K u \beta_K = u' \alpha^k$
\end{center}
for some $u'\in N'_{2k-1 +m_K}$. Hence, for any $u'_1 \in N'_{m_K}$, we have
\begin{center}
$[u'_1 \alpha^k, \beta_K u \beta_K v_0]= [u'_1 u'\alpha^k, v_0]$
\end{center}

By Lemma \ref{generators of sigma}, we may write any $v\in \sigma$ as a linear combination of the vectors $\{\beta_K u \beta_K v_0\}_{u \in N_{n_K}/ N_{n_K +1}}$. Putting the preceding together, the surjection of $T^- :R^- _k (\sigma)\rightarrow R^- _{k-1} (\sigma)$ follows immediately.

The surjection of $T^-: R^+ _n (\sigma)\rightarrow R^+ _{n-1} (\sigma)$ can be verified in the same way.

\medskip

If $T^+$ is not injective, its kernel is a non-zero $I_K$-stable space, thus it contains a non-zero $I_{1, K}$-invariant function (\cite[Lemma 1]{B-L95}). By \cite[Remark 3.8]{X2016}, such a function is proportional to $f_n$ or $f_{-n}$ ($n\geq 1$). We get a contradiction with (2) of Proposition \ref{hecke operator formula}.
\end{proof}

\begin{remark}
Note that by $(1)$ of Proposition \ref{image of the tree under T}, we may also define $T^-$ on the space $R^+ _0 (\sigma)$, but one can easily check it is not surjective anymore (see \cite[(4)]{X2016}).
\end{remark}

\begin{lemma}\label{reduction lemma on negative part of tree}
 Let $P(x)$ be a polynomial of degree at least one.

$(1)$.~~For $k\geq 0$, given an $f\in \oplus_{n\geq k}R^- _n(\sigma)$, there is an
 $f'\in \oplus_{n\geq k+1}R^{-}_{n}(\sigma)$, depending on $f$ and $P(x)$, such that
 \begin{center}
$f-f'\in P(T) (\oplus_{n\geq k+1}R^- _n(\sigma)).$
 \end{center}

$(2)$.~~For $k\geq 0$, given an $f\in \oplus_{n\geq k}R^+ _n(\sigma)$, there is an
 $f'\in \oplus_{n\geq k+1}R^+ _n (\sigma)$, depending on $f$ and $P(x)$, such that
 \begin{center}
$f-f'\in P(T) (\oplus_{n\geq k+1}R^+ _n (\sigma)).$
 \end{center}

 \end{lemma}

\begin{proof}
We only prove $(1)$ in detail, and the argument may be slightly modified to work for $(2)$.

We write $P(x)= (x-\lambda)P_1 (x)$ for some polynomial $P_1 (x)$ of degree strictly smaller than that of $P(x)$, and for
some $\lambda \in \overline{\mathbf{F}}_p$. By Corollary \ref{injectivity of T_- and surjectivity of T+}, we find some $g_1\in \oplus_{n\geq k+1}R^{-}_{n}(\sigma)$, such that $T^{-} (g_1)= f$.
If $P(x)$ is linear, the function $-T^{+} (g_1)+ \lambda g_1$ is as desired. If not, we do induction on the degree of $P(x)$. The induction hypothesis gives $g_2, g_3\in \oplus_{n\geq k+2}R^{-}_{n}(\sigma)$, such that $g_1- g_2= P_1 (T) (g_3)$. Now, the function
\begin{center}
 $-T^{+} (g_1)+ \lambda g_1 +(T-\lambda)g_2$
\end{center}
lies in $\oplus_{n\geq k+1}R^{-}_{n}(\sigma)$ and satisfies the requirement.
\end{proof}

\subsection{The action of $G$ on the spheres $C_{n, \sigma}$}

In this subsection, we carry out certain computations in full on the Bruhat--Tits tree of $G$. More specifically, we estimate the group action of $G$ on the pro-$p$-Iwahori invariants of a maximal compact induction. As we will see, it helps us to simplify and unify many later arguments.

We start with a simple and useful lemma:

\begin{lemma}\label{action of G on tree}

$(1)$.~For a function $[u \alpha^{-k}, v]\in R^+ _k (\sigma)$, where $u\in N_{n_K}, v\in \sigma$, and for $n\geq 0$, we have,

\begin{center}$\alpha^n \cdot[u \alpha^{-k}, v]=[\alpha^n \cdot u \alpha^{-k}, v]$
\end{center}
\begin{center}
$\begin{cases}
  \in R^+ _{k-n} (\sigma)\subset I^{+}(\sigma), ~~~~~\text{for} ~u\in N_{n_K +2n}, n\leq k;\\
  \in R^- _{k-n+t-1}\subset I^{-}(\sigma), ~~\text{for} ~u\in N_{n_K +2n-t}\setminus N_{n_K + 2n -t+1}, 1\leq t \leq 2n, n\leq k; \\
  \in R^- _{n-k+t-1}\subset I^{-}(\sigma), ~~~~~~\text{for} ~u\in N_{n_K +2k-t}\setminus N_{n_K+2k-t+1}, 0\leq t \leq 2k, ~k< n.
\end{cases}$
 \end{center}

$(2)$.~For a function $[u'\alpha^k, v]\in R^- _{k-1} (\sigma)$, where $u'\in N'_{m_K}, v\in \sigma$, and for $n\geq 0$, we have

\begin{center}$\alpha^{-n} \cdot [u' \alpha^k, v]=[\alpha^{-n} \cdot u' \alpha^k, v]$
\end{center}
\begin{center}
$\begin{cases}
  \in R^- _{k-n-1}(\sigma)\subset I^-(\sigma), ~~~~~~\text{for} ~u'\in N'_{2n+ m_K}, n< k;\\
  \in R^+ _{k-n+t-1} (\sigma)\subset I^+(\sigma), ~~\text{for} ~u'\in N'_{2n+m_K-t}\setminus N'_{2n+ m_K -t+1}, 1\leq t\leq 2n, n< k; \\
  \in R^+ _{n-k+t-1}(\sigma)\subset I^+(\sigma), ~~\text{for} ~u'\in N'_{2k+m_K-t}\setminus N'_{2k+ m_K -t+1}, 1\leq t \leq 2k, k\leq n; \\
  \in R^+ _{n-k}(\sigma)\subset I^+(\sigma), ~~~~~~~~~\text{for} ~u'\in N'_{2k+m_K}, k\leq n.
\end{cases}$
 \end{center}

$(3)$. The action of $\beta_K$ is given by:

$(a)$.~$\beta_K \cdot R^+ _0 (\sigma) = R^+ _0 (\sigma)$,~~$\beta_K \cdot R^+ _n (\sigma) \subseteq R^+ _n (\sigma)\oplus R^- _{n-1} (\sigma), n\geq 1$.

$(b)$.~$\beta_K \cdot R^- _{n-1} (\sigma) \subseteq R^+ _n (\sigma), n\geq 1$.

\end{lemma}

\begin{proof}
The statements in $(3)$ are straightforward from the definitions of $R^+ _n (\sigma)$ and $R^- _{n-1} (\sigma)$. We only say a few words for the second statement in $(a)$: for a function $f\in R^+ _n (\sigma)$ ($n\geq 1$), the matrix $\beta_K$ maps the part of $f$ supported in $K\alpha^n N_{n_K +1}$ into $R^- _{n-1} (\sigma)$, and the remaining part of $f$ into $R^+ _n (\sigma)$.

In the lists $(1)$ and $(2)$ above, only the second and third statements are not obvious, and they essentially follow from some explicit computation using the equality \eqref{useful identity}.
\end{proof}

\medskip

The main results of this part are summarized in the following two lemmas:
\begin{lemma}\label{transition}
$(1)$.~For $k\geq 0$, $n\geq 0$,
\begin{center}
$\alpha^n f_{-k}\in\begin{cases}
    f_{-(k-n)}+ \bigoplus_{k-n \leq m \leq k+n-1} R^- _m (\sigma), ~~~~n\leq k.\\
   \bigoplus_{n-k-1 \leq m \leq n+k-1} R^- _m (\sigma), ~~~~~n> k.
\end{cases}$
 \end{center}

$(2)$.~For $k\geq 1$, $n\geq 0$,
\begin{center}
$\alpha^{-n} f_k \in\begin{cases}
    f_{k-n}+ \bigoplus_{k-n \leq m \leq k+n-1} R^+ _m (\sigma) , ~~~~n< k. \\
   \bigoplus_{n-k \leq m \leq n+k-1} R^+_m (\sigma), ~~~~~n\geq k.
\end{cases}$
 \end{center}

\end{lemma}

\begin{proof}
 Mainly by the definitions of the functions $f_n$ ($n\in \mathbb{Z}$) and the first two lists in Lemma \ref{action of G on tree}.
\end{proof}

\begin{lemma}\label{transition2}
$(1)$.~For $k\geq 0$, $n\geq 0$,
\begin{center}
$\beta_K\alpha^n f_{-k}\in\begin{cases}
    f_{k-n}+\bigoplus_{k-n \leq m \leq k+n} R^+ _m (\sigma), ~~~~n< k.\\
    \bigoplus_{0 \leq m \leq 2n} R^+ _m (\sigma),            ~~~~n=k. \\
   \bigoplus_{n-k \leq m \leq n+k} R^+ _m (\sigma), ~~~~~n> k.
\end{cases}$
 \end{center}

$(2)$.~For $k\geq 1$, $n\geq 0$,
\begin{center}
$\beta_K\alpha^{-n} f_k \in\begin{cases}
   R^+ _k (\sigma), ~~~~n=0.\\
   f_{-(k-n)} +\bigoplus_{k-n \leq m \leq k+n-1} R^- _m (\sigma), ~~~~1\leq n\leq k.\\
   \bigoplus_{n-k-1 \leq m \leq n+k-1} R^- _m (\sigma), ~~~~~n> k.
\end{cases}$
 \end{center}

\end{lemma}

\begin{proof}
The lists in $(1)$ are straightforward from $(1)$ of Lemma \ref{transition} and $(3)$ of Lemma \ref{action of G on tree}, but notice here the subscript changes.

For $(2)$, note that $\beta_K\alpha^{-n}= \alpha^n \beta_K$. The function $\beta_K f_k$ ($k\geq 1$) is given by:
\begin{center}
$ \sum _{u'\in N'_{m_K} /N'_{m_K +2k-1}}~ [\beta_K u'\alpha^k, \beta_K v_0]= \sum_{u\in N_{n_K +1}/ N_{n_K +2k}} ~[u\alpha^{-k}, v_0].$
\end{center}
Especially, it lies in $R^+ _k (\sigma)$. Now the statements in $(2)$  results from $(1)$ of Lemma \ref{action of G on tree}.
\end{proof}

\section{The $I_K$-subrepresentation $I^+ (\sigma, \pi) \cap I^-(\sigma, \pi)$}\label{section: initial results on the I-reps}
We follow Hu's work (\cite{Hu12}) on canonical diagrams of $GL_2$ in this section. Roughly speaking, for a smooth representation $\pi$ of $G$ and a weight $\sigma$ of $K$ contained in $\pi$, we attach to $\pi$ an $I_K$-subrepresentation, and we verify some properties of such representation when $\pi$ satisfies some further conditions. We are very interested in such an $I_K$-subrepresentation, as in some sense it inherits important information from $\pi$. We remark that, in the case of $GL_2 (F)$, the analogue of such Iwahori subrepresentation is the key ingredient in Hu's canonical diagram.

\subsection{$I^+ (\sigma, \pi) \cap I^-(\sigma, \pi)\neq 0$ for certain $\pi$}\label{subsection 3.2}

Assume $\pi$ is a smooth representation of $G$, containing a weight $\sigma$ of $K$. By Frobenius reciprocity, there is an induced $G$-map $\iota$ from $\text{ind}^G _K \sigma$ to $\pi$.

By the Cartan decomposition $G= \bigcup_{n\geq 0} K\alpha^n K$, we have a decomposition of $\text{ind}^G _K \sigma$ into $K$-representations:
\begin{center}
$\textnormal{ind}^G _K \sigma= \oplus_{n\geq 0} R_n (\sigma)$
\end{center}

Also we have the $I_K$-decomposition of $\text{ind}^G _K \sigma$ as follows:
\begin{center}
$\text{ind}^G _K \sigma= I^+ (\sigma) \oplus I^- (\sigma)$
\end{center}
where $I^+ (\sigma)= \oplus_{n\geq 0} R^+ _n (\sigma)$, $I^- (\sigma)= \oplus_{n\geq 1} R^- _{n-1} (\sigma)$.

For an $f\in \textnormal{ind}^G _K \sigma$, denote by $\overline{f}$ the image of $f$ in $\pi$. Denote by $I^+ (\sigma, \pi)$ (resp, $I^- (\sigma, \pi), R^+ _n (\sigma, \pi), R^- _{n-1} (\sigma, \pi), R_n (\sigma, \pi)$) the image of $I^+ (\sigma)$ (resp, $I^- (\sigma), R^+ _n (\sigma), R^- _{n-1} (\sigma), R_n (\sigma)$) in $\pi$.

\begin{proposition}\label{non-vanishing of f_0}
Assume further that the $G$-map $\iota$ from $\textnormal{ind}^{G}_K \sigma$ to $\pi$ is surjective and factors through a quotient $\textnormal{ind}^{G}_K \sigma/(P(T))$ for some polynomial $P$ of degree $\geq 1$. Then

$(1)$.~$\overline{f_0}\in \sum_{n\geq 0} R^- _n (\sigma, \pi)$;

$(2)$.~$\overline{f_1}\in \sum_{n\geq 0} R^+ _n (\sigma, \pi)$.
\end{proposition}

\begin{proof}
It suffices to prove the following two statements:

$(1)'$.~~~$f_0\in P(T)(\textnormal{ind}^{G}_K \sigma)+ I^- (\sigma)$.

$(2)'$.~~~$f_1\in P(T)(\textnormal{ind}^{G}_K \sigma)+ I^+ (\sigma)$.

We start to prove $(1)'$. We pick a root $\lambda$ of $P(x)$ and write $P(x)= (x-\lambda)P_1 (x)$ for some polynomial $P_1 (x)$. Recall from $(1)$ of Proposition \ref{hecke operator formula}:
\begin{center}
$(T-\lambda) f_0= f_{-1} - \lambda f_0 +\lambda_{\beta, \sigma}f_1$.
\end{center}
We multiply both sides of above equality by $\alpha$, and we get:
\begin{center}
$(T-\lambda)\alpha f_0= \alpha f_{-1}- \lambda \alpha f_0 +\lambda_{\beta, \sigma}\alpha f_1$
\end{center}
Note that $\alpha f_0$ and $\alpha f_1$ lie in $I^- (\sigma)$. By (1) of Lemma \ref{transition}, the function $\alpha f_{-1} \in f_0 + I^- (\sigma)$. In all, we get
\begin{center}
$(T-\lambda)\alpha f_0 = f_0 +g_1$
\end{center}
for some function $g_1 \in I^- (\sigma)$.

 If $P_1 (x)$ is a constant, the preceding identity already gives us $(1)'$. Otherwise, using $(1)$ of Lemma \ref{reduction lemma on negative part of tree}, we find some
 $g_2\in \oplus_{n\geq 1}R^{-}_{n}(\sigma)$ such that
\begin{center}
 $\alpha f_0- g_2 \in P_1 (T) (\oplus_{n\geq 1}R^- _n (\sigma)),$
\end{center}
which gives that $f_0= (T-\lambda)g_2 -g_1 +P(T)f'$ for some $f'\in \oplus_{n\geq 1}R^- _n (\sigma)$, as desired for $(1)'$. We are done for $(1)$ .

We proceed to prove $(2)'$. Here we only need to prove in detail when $P$ is of degree one, and the general case follows from the same argument we have just done for $(1)'$, using $(2)$ of Lemma \ref{reduction lemma on negative part of tree}. Recall again that:
\begin{center}
$(T-\lambda)f_0= f_{-1}+ \lambda_{\beta_K, \sigma}f_1- \lambda f_0$.
\end{center}

By multiplying both sides of above equation by $\beta_K$, we get
\begin{center}
$(T-\lambda)\beta_K f_0= \beta_K f_{-1}+\lambda_{\beta_K, \sigma}\beta_K f_1 -\lambda \beta_K f_0$
\end{center}
Note that $\beta_K f_0 \in I^{+}(\sigma)$. By the first row in $(1)$ of Lemma \ref{transition2}, we have that $\beta_K f_{-1} \in f_1 +I^{+}(\sigma)$, whereas by the first row in $(2)$ of the same Lemma we have $\beta_K f_1 \in I^+ (\sigma)$. In summary, we get that
\begin{center}
$f_1\in (T-\lambda)+ I^+ (\sigma)$,
\end{center}
as desired.
\end{proof}

\begin{remark}\label{assumption on irreducibility}
The assumption on $\pi$ in the Proposition is at least satisfied in two cases: either $\pi$ is irreducible (\cite[Theorem 1.1]{X2018}) or is itself a spherical universal Hecke module.
\end{remark}

\begin{remark}
The Proposition says that the images of both functions $f_0$ and $f_1$ lie in $I^+ (\sigma, \pi)\cap I^- (\sigma, \pi)$. As the function $f_0$ generates $\textnormal{ind}^G _K \sigma$, under our assumption its image $\overline{f_0}$ in $\pi$ is non-zero, thus we have proved the representation $I^+ (\sigma, \pi)\cap I^- (\sigma, \pi)\neq 0$. However, for the function $f_1$, we can not say much about $\overline{f_1}$ at this stage, and we will address it elsewhere.
\end{remark}

Let $\phi_\sigma$ be the following $I_K$-homomorphism:
\begin{center}
$\phi_\sigma: \textnormal{ind}^{G}_K \sigma\twoheadrightarrow I^{-}(\sigma)\twoheadrightarrow I^{-}(\sigma,\pi)\hookrightarrow \pi,$
\end{center}
where the first surjection on the left is the natural projection from $\textnormal{ind}^{G}_K \sigma$ to $I^{-}(\sigma)$.

Denote by $R(\sigma, \pi)$ the kernel of $\iota$. Then, one has
\begin{lemma}\label{the image of phi_sigma}
 $I^{+}(\sigma, \pi)\cap I^{-}(\sigma, \pi)$ is the image of $R(\sigma, \pi)$ under $\phi_\sigma$.
\end{lemma}

\begin{proof}
 This key observation, even formal to check, is due to Y.Hu (\cite[Lemma 3.11]{Hu12}).
\end{proof}

\subsection{Finiteness of $R(\sigma, \pi) \Rightarrow \textnormal{dim}_{\overline{\mathbf{F}}_p}~I^+ (\sigma, \pi) \cap I^-(\sigma, \pi)< \infty$}
In this subsection we prove the following, which is the counterpart in our case of one side of Hu's criteria for finite presentation of smooth representations of $GL_2 (F)$.

Recall that we say $\pi$ is \emph{finitely presented}, if the $G$-representation $R(\sigma, \pi)$ is a finitely generated over $\overline{\mathbf{F}}_p [G]$.

\begin{proposition}\label{Hu's criteria}
Let $\pi$ be a smooth representation of $G$ and is a $G$-quotient of $\textnormal{ind}^{G}_K \sigma$.  Then the following condition $(2)$ implies $(1)$~$:$

$(1).$~~$I^{+}(\sigma, \pi)\cap I^{-}(\sigma, \pi)$ is of finite dimension~$;$

$(2).$~~$R(\sigma, \pi)$ is of finite type as a $\overline{\mathbf{F}}_p[G]$-module.
\end{proposition}

\begin{proof}

Assume $\{h_1, h_2, \cdots, h_l\}$ is a finite set in $R(\sigma, \pi)$ which generates it over $\overline{\mathbf{F}}_p [G]$.
For a large enough $m\geq 1$, all the $h_i$ lie in $\oplus_{0\leq k \leq m} R_k (\sigma)$. Let $M$ be the image of $\oplus_{0\leq k \leq m} R_k (\sigma)$ in $\pi$. By Lemma \ref{the image of phi_sigma}, we only need to show $\phi_{\sigma}(gh_i)\in M$ for all $g\in G$, as $M$ is of finite dimension.

Recall the Iwahori decomposition of $G$:
\begin{center}
$G= \bigcup_{g\in \mathcal{M}} I_K g I_K$
\end{center}
where $\mathcal{M}= \{\alpha^n, \beta_K \alpha^n\}_{n\in \mathbb{Z}}$. As the map $\phi_\sigma$ is an $I_K$-map, and the spaces $\oplus_{0\leq k \leq m} R_k (\sigma)\cap R(\sigma, \pi)$ and $M$ are both $I_K$-stable, we reduce us to the following Lemma:

\begin{lemma}\label{phi_sigma(alpha^n f)in M}
 For any $n\in \mathbb{Z}$, and any $f\in \oplus_{0\leq k \leq m} R_k (\sigma)\cap R(\sigma, \pi)$, both $\phi_\sigma (\alpha^n f)$ and $\phi_\sigma (\beta_K \alpha^n f)$ lie in $M$.
 \end{lemma}

\begin{proof}
We deal with case $n\geq 1$ in detail, and the remaining case $n< 0$ can be done in the same manner.

Note firstly that for $f\in I^{-}(\sigma)$, we have $\alpha^n f\in I^{-}(\sigma)$. By the first list of Lemma \ref{action of G on tree}, we see that for $f\in \oplus_{0\leq k \leq m}R_k (\sigma)$ and $n> m$, we have $(\alpha^n f^{+})^+ =0$, which gives $(\alpha^n f)^+=0$. When $n\leq m$, we also have $(\alpha^n f)^+ \in \oplus_{0\leq k \leq m}R_k (\sigma)$ by the same list. If $f$ is furthermore in $R(\sigma, \pi)$, we get $\phi_\sigma (\alpha^n f)= -\overline{(\alpha^n f)^+}\in M$ immediately.

We proceed to consider $\phi_\sigma (\beta_K \alpha^n f)$. Note that $\beta_K \alpha^n= \alpha^{-n}\beta_K$ and the matrix $\beta_K$ stabilizes the space $\oplus_{0\leq k \leq m} R_k (\sigma)\cap R(\sigma, \pi)$. We only need to consider $\phi_\sigma (\alpha^{-n}f)$. Similarly for $f\in I^+(\sigma)$, we have $\alpha^{-n} f\in I^+ (\sigma)$. By the second list of Lemma \ref{action of G on tree}, for $f\in \oplus_{0\leq k \leq m}R_k (\sigma)$ and $n\geq m$, we have $(\alpha^{-n} f^-)^- =0$, which gives $(\alpha^{-n} f)^-=0$. When $n< m$, we have $(\alpha^{-n} f)^- \in \oplus_{0\leq k \leq m}R_k (\sigma)$ by the same list. If $f$ is also in $R(\sigma, \pi)$, we get $\phi_\sigma (\alpha^{-n} f)= \overline{(\alpha^{-n} f)^-}\in M$.
\end{proof}

The argument of the proposition is now complete.
\end{proof}

\begin{remark}\label{observation from finiteness argument}
We have indeed proved the following: given an $f\in R(\sigma, \pi)$, let $m$ be the least integer such that
\begin{center}
$f\in \oplus_{0\leq k \leq m} R_k (\sigma)$.
\end{center}
Denote by $M$ the image of the space $\oplus_{0\leq k \leq m} R_k (\sigma)$ in $\pi$. Then, we have
\begin{center}
$\phi_\sigma (g\cdot f) \in M$, for any $g\in G$.
\end{center}
\end{remark}

\begin{remark}
  For the group $GL_2 (F)$, \textnormal{Hu} has indeed proved that the two conditions in the Proposition are equivalent (\cite[Theorem 4.3]{Hu12}), but in our case we are not able to prove $(1)$ implies $(2)$.
\end{remark}

\subsection{The $I_{1, K}$-invariant linear maps $S_K$ and $S_-$}
The purpose of this part is to study some partial linear operators on a smooth representation $\pi$, motivated by the Hecke operator $T$ (Proposition \ref{hecke operator formula}). We show that they satisfy some invariant properties, which will become useful in our later applications. We then study in detail how the $I_K$-morphism $\phi_\sigma$ (subsection \ref{subsection 3.2}) behaves with respect to such invariant linear operators.
\begin{definition}
Let $\pi$ be a smooth representation of $G$. We define:
\begin{center}
$S_K: \pi^{N'_{m_K}} \rightarrow \pi^{N_{n_K}}$,

$v \mapsto \sum_{u \in N_{n_K}/N_{n_K +1}}~u \beta_K v$.
\end{center}

\begin{center}
$S_-: \pi^{N_{n_K}} \rightarrow \pi^{N'_{m_K}}$,

$v\mapsto \sum_{u'\in N' _{m_K}/N' _{m_K +1}}~u'\beta_K\alpha^{-1}v$
\end{center}

\end{definition}

A simple check shows that both $S_K$ and $S_-$ are well-defined. We summarize the main properties of $S_K$ and $S_-$ as follows:
\begin{proposition}\label{S_1 and S_2}
We have:

$(1)$.~~Let $h\in H_1 =I_{1, K}\cap H$. Then $S_K (hv)= h^s\cdot S_K v$, for $v\in \pi^{N'_{m_K}}$, and $S_- (hv)= h^s\cdot S_- v$, , for $v\in \pi^{N_{n_K}}$, where $h^s$ is short for $\beta_K h \beta_K$.

$(2)$.~~If $v$ is fixed by $I_{1, K}$, the same is true for $S_K \cdot v$ and $S_- \cdot v$.
\end{proposition}

\begin{proof}
For $(1)$, we note that the group $H_1$ acts on $\pi^{N_{n_K}}$ and $\pi^{N'_{m_K}}$, as it normalizes $N_{n_K}$ and $N'_{m_K}$. The statement then follows from the definitions.

For $(2)$, we need some preparation, and we sort them out as two lemmas:

\begin{lemma}\label{exchange lemma}
For a $u'\in N'_{m_K}, u\in N_{n_K}$, we have:

$(1)$. ~The following identity
\begin{center}
$u'u= u_1 h u'_1$
\end{center}
holds for a unique $u_1\in N_{n_K}, h\in H_1, u'_1 \in N'_{m_K}$.

$(2)$. When $u$ goes through $N_{n_K}/ N_{n_K +m}$, the element $u_1$ also goes through $N_{n_K}/ N_{n_K +m}$, for any $m \geq 1$.
\end{lemma}

\begin{proof}
The uniqueness statement is clear, and only the existence needs to be proved.

Assume $u =n(x_1, y_1) \in N, u' \in n'(x, y) \in N'$. Then, if $1+x x_1 + \overline{y y_1}\in E^\times$, we have
 \begin{center}
 $u'u= u_1 h u'_1$，
\end{center}
where $h u'_1$ is the following lower triangular matrix:
\begin{center}
$\begin{pmatrix}  \frac{1}{1+x x_1 + \overline{y y_1}} & 0 & 0  \\ \frac{x-\overline{x_1 y}}{1+\overline{x x_1} +y y_1 } & \frac{1+x x_1+ \overline{y y_1}}{1+ \overline{x x_1} +y y_1} & 0\\
y & yx_1-\bar{x} & 1+ \overline{x x_1}+ y y_1
\end{pmatrix},$
\end{center}
and $u_1= n(x_2, y_2) \in N$, in which $x_2, y_2$ are given by:
\begin{center}
$x_2= \frac{x_1- \overline{y_1 x}}{1+x x_1+\overline{y y_1}}, y_2= \frac{y_1}{1+\overline{x x_1}+y y_1}.$
\end{center}

Under our assumption here, the condition $1+x x_1 + \overline{y y_1}\in E^\times$ holds automatically. The existence is established.

\medskip
We continue to prove $(2)$. We start by the following observation: from the formula of $y_2$ given in the argument of Lemma \ref{exchange lemma}, we see
\begin{center}
$y_2 =y_1 +$ higher valuation terms,
\end{center}
as $u=n(x_1, y_1)\in N_{n_K}, u'= n'(x, y)\in N'_{m_K}$. That is to say $u\in N_{n_K +m} \Leftrightarrow u_1\in N_{n_K +m}$ for any integer $m\geq 0$.

Assume now for an another $w\in N_{n_K}$, we have a decomposition $u'w= u_2 b''$ for $u_2\in N_{n_K}$ and $b''\in B'$. We have to prove:
\begin{center}
 $u_2\in u_1 N_{n_K +m}$ implies $w\in u N_{n_K +m}$.
\end{center}
 Write $u^{-1}_1 u_2$ as $u_3$. A little algebraic transform gives:
\begin{center}
$w= u \cdot b'^{-1}u_3 b''$
\end{center}
We need to check that the element $b'^{-1}u_3 b'' \in N_{n_K}$, denoted by $u_4$, lies in $N_{n_K +m}$. The element $b'$ can be written as $h\cdot u'_1$, for a diagonal matrix $h\in H_1$ and $u'_1 \in N'_{m_K}$. We therefore get
\begin{center}
$u'_1 u_4= (h^{-1}u_3 h)\cdot h^{-1}b''$,
\end{center}
where the right hand side is a decomposition of $u'_1 u_4$ given in last Lemma. The uniqueness of such a decomposition implies our observation at the beginning can be applied: we have
$u_4\in N_{n_K +m}$ iffy $h^{-1}u_3 h\in N_{n_K +m}$ for any $m\geq 0$. Our assumption is that $u_3= u^{-1}_1 u_2\in N_{n_K +m}$, which is the same as $h^{-1}u_3 h\in N_{n_K +m}$ ($h\in H_1$). We are done.
\end{proof}

\begin{lemma}\label{exchange lemma 2}
For a $u'\in N'_{m_K}, u\in N_{n_K}$, we have

$(1)$.~The following identity
\begin{center}
$u u'= u'_1 h u_1$
\end{center}
holds for a unique $u'_1 \in N'_{m_K}, h\in H_1, u_1\in N_{n_K}$.

$(2)$.~When $u'$ goes through $N'_{m_K}/ N'_{m_K +m}$, the element $u'_1$ also goes through $N'_{m_K}/ N'_{m_K +m}$, for any $m \geq 1$.
\end{lemma}

\begin{proof}
The argument of last Lemma can be slightly modified to work for the current case.
\end{proof}

\medskip
We proceed to complete the argument of $(2)$ of the Proposition.

By $(1)$ and the decomposition of $I_{1, K}= N'_{m_K} \times H_1 \times N_{n_K}$, it suffices to check that, for $u'= n'(x, y)\in N'_{m_K}$, the element $u'\cdot S_K v$
\begin{center}
$u'\cdot S_K v= \sum_{u\in N_{n_K}/N_{n_K +1}}~u'u\beta_K v$
\end{center}
is still equal to $S_K v=\sum_{u\in N_{n_K}/N_{n_K +1}}~u\beta_K v$. By $(1)$ of Lemma \ref{exchange lemma}, the right hand side of above sum is equal to:
\begin{center}
$\sum_{u\in N_{n_K}/N_{n_K +1}}~u_1 h u'_1 \beta_K v$.
\end{center}
 We get:
\begin{center}
$u'\cdot S_+ v=\sum_{u\in N_{n_K}/N_{n_K +1}}~u_1\beta_K (\beta_K h u'_1 \beta_K) v =\sum_{u\in N_{n_K}/N_{n_K +1}}~u_1\beta_K v$,
\end{center}
which is just the same as $\sum_{u_1\in N_{n_K}/N_{n_K +1}}~u_1\beta_K v$, by $(2)$ of Lemma \ref{exchange lemma}. The argument for the statement $S_K v \in \pi^{I_{1, K}}$ for $v\in \pi^{I_{1, K}}$ is complete now.

\medskip
Using Lemma \ref{exchange lemma 2}, the previous argument can be slightly modified to work for the statement $S_- v \in \pi^{I_{1, K}}$ for $v\in \pi^{I_{1, K}}$.

We are done for the Proposition.
\end{proof}

\medskip
Recall that the space $R^+ _n (\sigma)^{I_{1, K}}$ ($n\geq 0$) and $R^- _{n-1} (\sigma)^{I_{1, K}}$ ($n\geq 1$) are both one-dimensional (\cite[Remark 3.8]{X2016}), and we will improve it slightly here, as an application of the stuff we have just carried out:

\begin{proposition}\label{N-invariants of tree}
We have:

 $(1)$.~For $n\geq 0$, $R^+ _n (\sigma)^{N_{n_K}}= \langle f_{-n}\rangle_{\overline{\mathbf{F}}_p}$.

$(2)$.~For $n\geq 1$, $R^- _{n-1} (\sigma)^{N'_{m_K}}= \langle f_n \rangle_{\overline{\mathbf{F}}_p}$.
\end{proposition}

\begin{proof}
We only prove $(1)$ in detail, and the argument for $(2)$ is completely parallel.

Note that the group $N_{n_K}$ is only a closed subgroup of $I_{1, K}$. Let $f$ be a non-zero function in the space $R^+ _n (\sigma)^{N_{n_K}}$. We claim that $f$ in indeed fixed by the group $I_{1, K}$. Note that $K\alpha^n I_K= K\alpha^n N_{n_K}$ for $n\geq 0$, hence the function $f$ is determined by $f(\alpha^n u)$ for all $u\in N_{n_K}$. For a $b'\in N'_{m_k}\times H_1$, we have
\begin{center}
$b'\cdot f (\alpha^n u)= f(\alpha^n u b')= f(\alpha^n b'_1 u_1)$,
\end{center}
for some $b'_1 \in N'_{m_k}\times H_1, u_1\in N_{n_K}$, where we have used $(1)$ of Lemma \ref{exchange lemma 2} for the second equality. We now simply have that (by definition and the assumption on $f$):
\begin{center}
$f(\alpha^n b'_1 u_1)=f(\alpha^n)=f(\alpha^n u)$
\end{center}
We have proved $f$ is fixed by the group $I_{1, K}\cap B'$, hence the claim. We are done.
\end{proof}

\subsection{Is $I^{+}(\sigma, \pi)\cap I^{-}(\sigma, \pi)$ is canonical ?}\label{subsection: canonical}

Let $\pi$ be an irreducible smooth representation of $G$. For a weight $\sigma$ of $K$ contained in $\pi$, we have attached to $\pi$ an $I_K$-subrepresentation $I^{+}(\sigma, \pi)\cap I^{-}(\sigma, \pi)$ and proved it is non-zero (Proposition \ref{non-vanishing of f_0}). In this subsection, we prove the following conditional result:

\begin{proposition}\label{canonicalness}
 Assume $\pi^{I_{1, K}}\subseteq I^{+}(\sigma, \pi)\cap I^{-}(\sigma, \pi)$ holds. Then the $I_K$-subrepresentation $I^{+}(\sigma, \pi)\cap I^{-}(\sigma, \pi)$ of $\pi$ does not depend on the choice of $\sigma$.
\end{proposition}

\begin{remark}
Our assumption on $\pi$ made in the Proposition is quite awkward. Actually in the case of $GL_2$, it is the major input Hu has arrived to prove his diagram is canonical (\cite[Proposition 3.16]{Hu12}). In our case, due to some technical reason, we are not able to prove it at this stage.
\end{remark}

Denote by $P^+$ and $P^-$ respectively the following semigroups in $G$:
\begin{center}
$P^+ := N_{n_K}\alpha^{-\mathbf{\mathbb{Z}}_{\geq 0}}, ~P^- := N_{n_K +1}\alpha^{-\mathbf{\mathbb{N}}}.$
\end{center}
Note that the semigroup $P^-$ does not contain $Id$, and it is \emph{properly contained} in $P^+$.

A simple computation using Lemma \ref{generators of sigma} on the spaces $I^+ (\sigma)= \oplus_{n\geq 0} R^+ _n (\sigma)$ and $I^- (\sigma)= \oplus_{n\geq 1} R^- _{n-1} (\sigma)$ gives that:
\begin{lemma}\label{unifom of I^+ and I^-}
$I^+ (\sigma)= [P^+ \beta_K, v_0],~ I^- (\sigma)= [\beta_K P^- \beta_K, v_0]$.
\end{lemma}

\begin{proof}[Proof of Proposition \ref{canonicalness}]

By Lemma \ref{unifom of I^+ and I^-}, we have that
\begin{center}
$I^+ (\sigma, \pi)= \langle P^+ \beta_K \overline{[Id, v_0]}\rangle_{\overline{\mathbf{F}}_p}, ~I^- (\sigma, \pi)= \langle \beta_K P^- \beta_K \overline{[Id, v_0]}\rangle_{\overline{\mathbf{F}}_p}$,
\end{center}
that is $v\in I^+ (\sigma, \pi)$ if and only if there is a $Q\in \overline{\mathbf{F}}_p [P^+]$ such that
\begin{center}
$v= Q \beta_K \overline{[Id, v_0]}$
\end{center}
Similarly, we have $v\in I^- (\sigma, \pi)$ if and only if there is a $Q\in \overline{\mathbf{F}}_p [P^-]$ such that:
\begin{center}
$v= \beta_K Q \beta_K \overline{[Id, v_0]}$
\end{center}

Now for another $\sigma'$ contained in $\pi$, let $w_0$ be a non-zero vector in the line $\sigma'^{I_{1, K}}$. Note that $\overline{[Id, w_0]} \in \pi^{I_{1, K}} \subseteq I^{+}(\sigma, \pi)\cap I^{-}(\sigma, \pi)$ by the assumption. By the preceding remarks, we find $Q_1 \in \overline{\mathbf{F}}_p [P^+]$ and $Q_2 \in \overline{\mathbf{F}}_p [P^-]$ such that:
\begin{center}
$\overline{[Id, w_0]}= Q_1 \beta_K \overline{[Id, v_0]}= \beta_K Q_2 \beta_K \overline{[Id, v_0]}.$
\end{center}

By Lemma \ref{unifom of I^+ and I^-} again,
\begin{center}
$I^+ (\sigma', \pi)= \langle P^+ \beta_K \overline{[Id, w_0]}\rangle= \langle P^+ \beta^2 _K  Q_2 \beta_K \overline{[Id, v_0]}\rangle \subseteq \langle P^+ \beta_K \overline{[Id, v_0]}\rangle $,
\end{center}
where we note that $\beta^2 _K= Id$, and $Q_2 \in \overline{\mathbf{F}}_p [P^-]\subset \overline{\mathbf{F}}_p [P^+]$. We therefore have verified one side inclusion:
\begin{center}
$I^+ (\sigma', \pi) \subseteq I^+ (\sigma, \pi)$.
\end{center}
By exchanging the role of $\sigma$ and $\sigma'$, the same argument gives the other side inclusion:
\begin{center}
$I^+ (\sigma, \pi) \subseteq I^+ (\sigma', \pi)$.
\end{center}
Hence, we have $I^+ (\sigma, \pi) = I^+ (\sigma', \pi)$.

Almost the same argument gives $I^- (\sigma, \pi) = I^- (\sigma', \pi)$. We are done.
\end{proof}

\medskip

\section{Some computation on $I^{+}(\sigma, \pi)\cap I^{-}(\sigma, \pi)$: examples}\label{section: main results}
In principle, it is hard to determine the $I_K$-subrepresentation $I^{+}(\sigma, \pi)\cap I^{-}(\sigma, \pi)$ of $\pi$, for a general $\pi$ and an underlying $\sigma$. However, when the corresponding kernel $R(\sigma, \pi)$ is known in advance, it is possible to detect it via Lemma \ref{the image of phi_sigma}. We explore such a point in this final section.

\medskip

\subsection{The case that $\pi$ is a spherical universal Hecke module}

In this part, we study the $I_K$-subrepresentation attached to a spherical universal Hecke module, i.e., a $G$-representation of the form $\textnormal{ind}^G _K \sigma /(P(T))$ for some polynomial $P$ of degree $\geq 1$. Such a $G$-representation plays a central role in the $p$-modular representation theory of $G$, but we don't know much about it in general\footnote{We do know such a representation is always infinite dimensional (\cite[Corollary 4.6]{X2016}).}. We prove an analogue of Hu's result on $GL_2 (F)$ (\cite[Proposition 3.13]{Hu12}), but our argument here is almost formal, based on the computation carried out in Section \ref{section: spheres}. We then manage to give it a canonical basis when the polynomial $P$ is linear.

\begin{proposition}\label{I-space contained in I_1 for coker(P(T))}
Let $\pi$ be the representation $\textnormal{ind}^G _K \sigma /(P(T))$, for some weight $\sigma$ of $K$, and some polynomial $P(x)$ of degree $\geq 1$. Then the inclusion $I^{+}(\sigma, \pi)\cap I^{-}(\sigma, \pi)\subseteq \pi^{I_{1,K}}$ holds.
\end{proposition}

\begin{proof}
As $P(T)\textnormal{ind}^G _K \sigma= \langle P(T)f_0 \rangle_G$, by Lemma \ref{the image of phi_sigma} it suffices to prove that $\phi_\sigma (g\cdot P(T)f_0)\in \pi^{I_{1,K}}$, for any $g\in G$. Recall that $\phi_\sigma$ is $I_K$-linear, and the group $I_K$ acts as a character on the function $f_0$. By the Iwahori decomposition of $G$:
\begin{center}
$G= \bigcup_{g\in \mathcal{M}} I_K g I_K$
\end{center}
where $\mathcal{M}= \{\alpha^n, \beta_K \alpha^n\}_{n\in \mathbb{Z}}$, it is enough to verify the former statement for all $g \in \mathcal{M}$.

Assume $P(x)$ is of degree $\geq 1$. By Proposition \ref{hecke operator formula}, $P(T)f_0$ is just a linear combination of the functions $\{f_k\}_{k\in \mathbb{Z}}$. Note that $P(T)f_0=  (P(T)f_0)^+ +(P(T)f_0)^-$. For $n\geq 1$, we have $\alpha^{-n} (P(T)f_0)^+\in I^+(\sigma)$. By the second list in Lemma \ref{transition}, we have
\begin{center}
$\alpha^{-n} (P(T)f_0)^- =\sum_{k\geq 1} c_k f_k  +f$
\end{center}
for some $f\in I^+ (\sigma)$. Hence, we have that
\begin{center}
$\phi_\sigma (\alpha^{-n}P(T)f_0) = \overline{\sum_{k\geq 1} c_k f_k}$,
\end{center}
which is certainly in $\pi^{I_{1, K}}$, as required. Almost the same argument using the first list in Lemma \ref{transition} gives that: for $n\geq 1$
\begin{center}
$\phi_\sigma (\alpha^n P(T)f_0)= -\overline{\sum_{k\geq 0} c_k f_{-k}}$.
\end{center}

It remains to verify that $\phi_\sigma(\beta_K\alpha^n P(T)f_0)\in \pi^{I_{1, K}}$, for any $n\in \mathbb{Z}$. But this follows from the same idea where we just need to apply the lists in Lemma \ref{transition2}. Here, we record the key details as follows:
\begin{center}
For $n\geq 0$,  we have $\phi_\sigma (\beta_K \alpha^n P(T)f_0)= \overline{\sum_{k\geq 1} c_k f_k}$.

For $n\geq 1$,  we have $\phi_\sigma (\beta_K \alpha^{-n} P(T)f_0)= -\overline{\sum_{k\geq 0} c_k f_{-k}}$.
\end{center}
Note that $\beta_K \alpha^{-n} f_{-m} \in I^- (\sigma)$ ($m\geq 0$) in the second case above (see the argument of Lemma \ref{action of G on tree}).
\end{proof}

\begin{remark}\label{new observation}
The observation underlying our argument is that, for a function $f\in (\textnormal{ind}^G _K \sigma)^{I_{1, K}}$ and a $g\in G\setminus I_K$, one of the two functions $(g\cdot f)^+$ and $(g\cdot f)^-$ (possibly zero function) is still $I_{1, K}$-invariant, even $g\cdot f$ is \textbf{not}.
\end{remark}

\begin{remark}
The representation $\pi$ considered in the Proposition is certainly finitely presented, and Proposition \ref{Hu's criteria} tells that the space $I^{+}(\sigma, \pi)\cap I^{-}(\sigma, \pi)$ is finite dimensional. When $P(T)$ satisfies some further condition, the representation $\pi$ might not be admissible. That is to say, for an arbitrary $\pi$, the space $I^{+}(\sigma, \pi)\cap I^{-}(\sigma, \pi)$ being finite dimensional does not imply the admissibility of $\pi$, which, however, might be true under the further assumption that $\pi$ is irreducible (as in the case $GL_2 (F)$, see \cite[Proposition 3.16]{Hu12}).
\end{remark}

When the polynomial $P$ is linear, we may say a little more:
\begin{theorem}\label{images of f_0 and f_1 give that of I_K}
Assume $\pi$ is the representation $\textnormal{ind}^G _K \sigma / (T-\lambda)$. Then the $I_K$-subrepresentation $I^{+}(\sigma, \pi)\cap I^{-}(\sigma, \pi)$ is two dimensional, with a basis $\{\overline{f_0}, \overline{f_1}\}$.
\end{theorem}

\begin{proof}
Recall that $\overline{f_0}$ and $\overline{f_1}$ are linearly independent in $\pi$ (Corollary \ref{f_0 and f_1 are l.d in quotient}). By Proposition \ref{non-vanishing of f_0}, both $\overline{f_0}$ and $\overline{f_1}$ lie in $I^+ (\sigma, \pi)\cap I^- (\sigma, \pi)$.

As $\phi_\sigma$ is an $I_K$-map, and the group $I_K$ acts on the functions $f_k$ ($k\in \mathbb{Z}$) as characters,  to complete the proof we again reduce us to verify that, for the function $f= (T-\lambda)f_0$, the following
\begin{center}
$\phi_\sigma (g\cdot f)\in \langle \overline{f_0}, \overline{f_1}\rangle_{\overline{\mathbf{F}}_p}$
\end{center}
holds for all $g\in \mathcal{M}= \{\alpha^n, \beta_K \alpha^n\}_{n\in \mathbb{Z}}$. But the argument of last Proposition works here, and from that we only need Corollary \ref{f_0 and f_1 are l.d in quotient} to get the above claim. We have proved the other side inclusion of the statement.
\end{proof}

\medskip
\subsection{The case that $\pi$ is an irreducible principal series}\label{subsection: irreducible principal series}

For an irreducible smooth representation $\pi$, and an underlying irreducible smooth representation $\sigma$ of $K$, we have proved \emph{conditional} in subsection \ref{subsection: canonical} that the $I_K$-subrepresentation $I^+ (\sigma, \pi) \cap I^-(\sigma, \pi)$ does not depend on the choice of $\sigma$.

Nevertheless, we may determine it with ease when $\pi$ is an irreducible principal series.

\medskip
When we say $\pi$ is an \emph{irreducible principal series}, we mean it is in one of the following three cases.

$(i)$.~~$\chi\circ \text{det}$ for any character $\chi$ of $E^1$;

$(ii)$.~~$\chi\circ \text{det}\otimes St$ for any character of $E^1$, where $St$ is the Steinberg representation $\text{ind}^G _B 1 /1$.

$(iii)$.~~$\text{ind}^G _B \varepsilon$, for any character $\varepsilon$ of $B$ which does not factor through the determinant.

\begin{theorem}\label{main}
Assume $\pi$ is an irreducible principal series, containing a weight $\sigma$ of $K$. Then the $I_K$-subrepresentation $I^{+}(\sigma, \pi)\cap I^{-}(\sigma, \pi)$ is equal to $\pi^{I_{1, K}}$.
\end{theorem}

\begin{proof}
The statement is trivial when $\pi$ is in case $(i)$. Assume $\pi$ is in case $(iii)$. There is a unique (up to a scalar) and explicit $G$-homomorphism $P$ from $\textnormal{ind}^G _K \sigma$ to $\pi$ with kernel $(T-\lambda)$ for some scalar $\lambda$ (\cite{AHHV17}). Therefore, we only need to understand $\phi_\sigma ((T-\lambda))$. Now the argument of Theorem \ref{images of f_0 and f_1 give that of I_K} works completely the same here, so the former space is spanned by the vectors $P(\overline{f_0})$ and $P(\overline{f_1})$, which is nothing but the two dimensional subspace of $I_{1, K}$-invariants of $\pi$. Note that the $\sigma$ is chosen arbitrarily underlying $\pi$.

When $\pi$ is in case $(ii)$, we may assume $\chi$ is trivial, i.e., $\pi$ is $St$. There is also an explicit and unique (up to scalar) $G$-homomorphism $P$:
\begin{center}
$P: \textnormal{ind}^G _K  \text{st} \rightarrow St$,
\end{center}
with kernel $(T)\oplus \langle f_0+ f_1 \rangle_{\overline{\mathbf{F}}_p}$(\cite{AHHV17}), where $st$ is the Steinberg weight of $K$. We know the non-zero vector $P(\overline{f_0})$ generates the unique line $St^{I_{1, K}}$. Certainly we have that $P(\overline{f_0})\in I^+(st, St)\cap I^-(st, St)$. It suffices to verify the following, where we note that $Tf_0= f_{-1}$ ($(1)$ of Proposition \ref{hecke operator formula} and Remark \ref{value of lambda}) in this case:
\begin{center}
$\phi_{st}(g (c_{-1}f_{-1}+c_1 (f_0+f_1)))\in \langle P(\overline{f_0})\rangle_{\overline{\mathbf{F}}_p},$
\end{center}
for any $c_{-1}, c_1\in \overline{\mathbf{F}}_p$, and for all $g\in \mathcal{M}= \{\alpha^n, \beta_K \alpha^n\}_{n\in \mathbb{Z}}$. As in Theorem \ref{images of f_0 and f_1 give that of I_K}, we may apply the argument of Proposition \ref{I-space contained in I_1 for coker(P(T))}, and the claim follows from Corollary \ref{f_0 and f_1 are l.d in quotient}, where we note $\overline{f_0}= -\overline{f_1}$ in the current case. We are done.
\end{proof}

\begin{remark}\label{diagram: definition}
Let $\pi$ be a smooth representation of $G$. Take an irreducible smooth representation $\sigma$ of $K_0$ underlying $\pi$. We may then attach a diagram $\mathcal{D}(\pi, \sigma)$ (\cite[6.2]{Karol-Peng2012}) to $\pi$ as follows:
\begin{center}
$(D_0, D_1, I^+ (\sigma, \pi)\cap I^- (\sigma, \pi), r_0, r_1),$
\end{center}
in which $D_i$ is the $K_i$-subrepresentation of $\pi$ generated by $I^+ (\sigma, \pi)\cap I^- (\sigma, \pi)$, and $r_i$ is the inclusion map from $I^+ (\sigma, \pi)\cap I^- (\sigma, \pi)$ to $D_i$ ($i=0, 1$).

When $\pi$ is an irreducible principal series, based on Theorem \ref{main} we may prove:
\begin{center}
$\mathcal{D}(\pi, \sigma)= (\pi^{K^1 _0}, \pi^{K^1 _1}, \pi^{I_{1, K_0}}, r_0, r_1)$.
\end{center}

\end{remark}

\section*{Acknowledgements}
Part of this paper was initiated from the author's PhD thesis (\cite{X2014}) at University of East Anglia, and our debt owned to the beautiful work of Yongquan Hu (\cite{Hu12}) should be very clear to the readers. The author was supported by a Postdoc grant from Leverhulme Trust RPG-2014-106 and European Research Council project 669655.
\medskip

\bibliographystyle{amsalpha}
\bibliography{new}

\texttt{Einstein Institute of Mathematics, HUJI, Jerusalem, 9190401, Israel}

\emph{E-mail address}: \texttt{Peng.Xu@mail.huji.ac.il}

\end{document}